\documentclass{amsproc}
\usepackage{amsmath,amsthm,amsfonts,amscd, amssymb, epsf}
\usepackage{geometry}
\makeatletter
 \theoremstyle{plain}
\newtheorem{thm}{Theorem}[section]
  \theoremstyle{plain}
  \newtheorem{lem}[thm]{Lemma}
  \newtheorem{prop}[thm]{Proposition}
  \newtheorem*{thm*}{Theorem}
  \newtheorem*{rem*}{Remark}
\numberwithin{equation}{section}
\usepackage{palatino}
\usepackage{amsthm}
\usepackage{amsfonts}
\usepackage{amscd}
\usepackage{epsf}
\makeatletter

\newcommand{\field}[1]{\ensuremath{\mathbb{#1}}}

\newcommand{\hh}{\field{H} \,}

\newcommand{\RR}{\field{R}}

\newcommand{\ZZ}{\field{Z}}

\newcommand{\E}{\mathcal{E}}

\DeclareMathOperator{\pr}{PSL(2,\RR)}

\makeatother

\begin{document}

\author{Joshua S. Friedman}\footnote{The views expressed in this article are the author's
own and not those of the U.S. Merchant Marine Academy, the Maritime Administration,
the Department of Transportation, or the United States government.}
\address{ Department of Mathematics and Science, United States Merchant Marine Academy, 300 Steamboat Road, Kings Point, NY  11024}
\email{friedmanj@usmma.edu}

\title[The distance between elliptic fixed points]{On the minimal distance between elliptic fixed points for geometrically-finite Fuchsian groups}

\begin{abstract}
Let $\Gamma$ be a geometrically-finite Fuchsian group acting on the upper half plane $\hh.$ Let $\E$ denote the set of elliptic fixed points of $\Gamma$ in $\hh.$ We give a lower bound on the minimal hyperbolic distance between points in  $\E.$ Our bound depends on a universal constant and the length of the smallest closed geodesic on $\Gamma \backslash \hh.$ 
\end{abstract}

\maketitle

\thispagestyle{empty}
\section{Introduction}
Let $\Gamma$ be a geometrically-finite Fuchsian group (see \cite[Chap. 10]{Beardon} or \cite[Chap. 4]{Katok}) acting on the upper half plane $\hh.$ The finiteness conditions guarantees that $\Gamma$ contains a finitely many conjugacy classes of maximal elliptic cyclic subgroups. Let $\E$ denote the set of elliptic fixed points of $\Gamma$ in $\hh.$ The set $\E$ has no accumulation points since $\Gamma$ is discrete. The stabilizer of each point of $\E$ corresponds to a finite cyclic group generated by a \emph{primitive} elliptic element. An interesting problem is to determine the minimal hyperbolic distance between any two elements of $\E,$ in terms of fundamental geometric qualities of the quotient orbifold $\Gamma \backslash \hh,$ such as its area, genus, number of cusps, funnels, or the length of its smallest geodesic (systole). 

Our main result is the following:
\begin{thm*}
Let $\Gamma$ be a geometrically-finite group. Let $\E$ denote the set of elliptic fixed points of $\Gamma$ in $\hh.$ Let $l_{0}$ be the length of the smallest closed geodesic on $\Gamma \backslash \hh.$ Let $z,w$ be chosen in $\E$ so that the hyperbolic distance $\rho(z,w)$ is minimal. Then
$$ \rho(z,w) \geq \min\{\tfrac{l_{0}}{2}, 0.151\dots \}.$$

If all elliptic fixed points have order greater than two, the $\tfrac{l_{0}}{2}$ term can be omitted. 

\end{thm*}

Randol \cite{Randol2} gave a universal lower bound on the area of a cylinder around a closed geodesic in a Riemann surface. One can consider our result as an analogue, giving a universal lower bound for the area of disjoint neighborhoods around elliptic cusps. 

I would like to thank Professors Bernie Maskit and Jay Jorgenson for answering questions and helpful suggestions. 

\section{Preliminaries}
Let $\hh = \{z\in\mathbb{C}\,|\,z=x+iy\,,\,y>0\}$ with the hyperbolic metric $\mathrm{d}s^{2}(z)=\frac{\mathrm{d}x^{2}+\mathrm{d}y^{2}}{y^{2}}.$ Let $\rho(z,w)$ denote the hyperbolic distance between two points in $\hh.$ Each element $\gamma \in \pr$ acts on $\hh$ via the M\"{o}bius transformation $f(z) = \frac{az+b}{cz+d}.$ Let $\Gamma$ be a geometrically-finite group. An elliptic element $\gamma$ is determined by its unique fixed point $z \in \hh,$ and a rotation angle $\theta \in [-\pi,\pi]$ about $z.$ To each point in $z \in \E$ is an primitive elliptic element of minimal rotation angle $2\pi/n$ which generates the stabilizer subgroup $\Gamma_{z},$ where $n$ is the order of the subgroup. 

A hyperbolic element $\alpha$ is conjugate in $\pr$ to $g(z) = kz,$ ( $k>1$ ). Associated to $\alpha$ is an \emph{axis,} a geodesic in $\hh$ where $\alpha$ acts as a translation with length $T_{\alpha} = \ln(k) = \inf_{z \in \hh}\rho(z,\alpha z).$ The image of this axis in the quotient orbifold $\Gamma \backslash \hh$ is a closed geodesic of length $T_{\alpha}.$

Let $\Gamma$ be a geometrically-finite Fuchsian group acting on the upper half plane $\hh.$ A subgroup $\Gamma_{0}$ is \emph{elementary} if it has a finite orbit in the closure $\overline{\hh}.$ These groups are completely characterized. In the following lemma, we collect and prove some basic facts that will be needed.

\begin{lem} \label{lemMain}
\end{lem}

\begin{enumerate}

\item An elementary Fuchsian group is either cyclic or conjugate in $\pr$ to the group $\left<h,e \right>$ generated by $h(z) = kz,$ ($k>1$) and $e(z)=-1/z.$  

\item \label{lemElem} Let $A,B \in \pr$ be elliptic elements, with different fixed points $z,w$ respectively, which generate an elementary Fuchsian group. Then $A$ and $B$ both have order two. Furthermore, $AB$ is hyperbolic and the translation length $T_{AB} = 2\rho(z,w).$

\item \label{lemOrTwo} Let $A,B \in \pr$ be elliptic elements of order two. Then $\left<A,B\right>$ is elementary.  

\item \label{lemDisp} Let $\gamma$ be elliptic with fixed point $v$ and angle of rotation $\theta.$ Then for every $z \in \hh,$ 
$$\sinh \tfrac{1}{2}\rho(z,\gamma z) = \sinh \rho(z,v) |\sin(\theta/2)|. $$ 
\end{enumerate}
\begin{proof}
(1) Is found in \cite[p. 37]{Katok}. (4) is found in \cite[Theorem 7.25.1]{Beardon}. 

(2) By (1) $\left< A, B \right>$ is conjugate to $G = \left< e, h \right>.$ Note that $e = e^{-1}, $ $ehe^{-1} = h^{-1},$ and $eh = h^{-1}e.$ It follows that any element of $G$ is of the form $h^{k}$ or $eh^{k},$ where $k \in \ZZ.$ The elements $h^{k}$ are all hyperbolic when $k\neq 0,$ while the $eh^{k}$ are all elliptic of order two since $eh^{k}eh^{k} = eh^{k}h^{-k}e = e^{2} = I.$ Hence $A$ and $B$ must be of order two. Let $A$ be conjugate to $eh^{l},$  $B$ to $eh^{m}.$ Then $AB$ is conjugate to $eh^{l}eh^{m} = h^{-l}eeh^{m} = h^{m-l}$ and is hyperbolic. 

To prove the second statement, by conjugation, assume $A = e,$ $B = f= eh.$ The fixed point of $e(z)$ is $i,$ while the fixed point of $f(z)$ is $\frac{i}{\sqrt{k}}.$ Both points are on the axis of $h(z)$ and since $h(\frac{i}{\sqrt{k}}) = \sqrt{k}i$ it follows that $T_{h} = \rho(\frac{i}{\sqrt{k}},\sqrt{k}i)$ which is twice the distance between the two fixed points of $A$ and $B.$ 

(3) Let $z$ and $w$ be the respective fixed points of $A$ and $B.$ If $z=w$ we are done, so assume $z \neq w.$ Choose $\gamma \in \pr$ with $\gamma(z) = i, \gamma(w) = \tfrac{i}{\sqrt{k}}.$ Now conjugate    $\left<A,B\right>$ with $\gamma,$ reverse the argument in (2) and apply (1).
\end{proof}

For more information about when two generator subgroups are discrete, and for an algorithm to decide the question, see \cite{Maskit}.

Using the above lemma and an elementary argument we obtain 
\begin{prop}
Let $\Gamma$ be a geometrically-finite group let $\E$ be as defined above. 
Let $z,w$ be chosen in $\E$ so that $\rho(z,w)$ is minimal. Let $A,B$ be the respective primitive elliptic elements of $\Gamma$ that fix $z,w$ and $\phi, \theta$ be their respective rotation angles. Assume $\theta \leq \phi.$ Then $$2 \cosh\tfrac{1}{2}\rho(z,w) \geq \frac{1}{|\sin(\theta/2)|}. $$
\end{prop}

\begin{proof}
By the above lemma, $\sinh \tfrac{1}{2}\rho(z,Bz) = \sinh \rho(z,w) |\sin(\theta/2)|.$ Since $\rho(z,w)$ is minimal and since $Bz$ is an elliptic fixed point of the elliptic element $BAB^{-1},$ it follows that $\sinh \tfrac{1}{2}\rho(z,Bz)  \geq \sinh \tfrac{1}{2}\rho(z,w).$  Thus $\sinh \rho(z,w) |\sin(\theta/2)| \geq \sinh \tfrac{1}{2}\rho(z,w).$  But $$\frac{\sinh \rho(z,w)}{\sinh \tfrac{1}{2}\rho(z,w)} = 2 \cosh\tfrac{1}{2}\rho(z,w).  $$ 
\end{proof}

Note that if $\theta = \pi/3,$ the lemma tells us that $2 \cosh\tfrac{1}{2}\rho(z,w)\geq 2$ which is not meaningful.  It is only meaningful when $\theta < \pi/3.$ To get to the main result, which is meaningful in all cases, we need the following:

Let $g,h \in \pr$ define 
$$M(g,h) = \inf_{z \in \hh} m(z) =  \max\{\sinh \tfrac{1}{2}\rho(z,g(z)), \sinh \tfrac{1}{2} \rho(z,h(z)) \}.  $$

Marden established the existence of a universal lower bound on $M$  while Yamada found the best possible lower bound (\cite{Marden, Yamada}).
 
\begin{lem}\cite[p. 313]{Beardon} \label{lemMY}
Let $g,h$ be elliptic elements. Suppose $\left<g,h \right>$ is discrete and non-elementary. Then 
$$M(g,h) \geq C = \left(\frac{4\cos^{2}(\pi/7)-3}{8\cos(\pi/7)+7}    \right)^{1/2} = 0.1318\dots. $$
\end{lem}

When the two-generator elliptic group generates a non-elementary group we have the following: 

\begin{lem}\label{lemNonElem}
Let $g,h$ be primitive elliptic elements with respective fixed points $z,w$, and let $\phi, \theta$ be their respective rotation angles. Assume $\theta \leq \phi.$ Suppose $\left<g,h \right>$ is discrete and non-elementary. Then $$\sinh \rho(z,w) \geq \frac{C}{|\sin(\theta/2)|}. $$ \end{lem}
\begin{proof}
Since $z$ is fixed by $g,$ $\rho(z,g(z)) = 0,$ so by Lemma\ref{lemMain}~(\ref{lemDisp}) and Lemma~\ref{lemMY},
$$m(z) = \sinh \tfrac{1}{2} \rho(z,h(z)) = \sinh \rho(z,w) |\sin(\theta/2)| \geq C. $$ 
\end{proof}

We can now prove the main result:
\begin{proof}
Let $z,w$ be chosen in $\E$ so that $\rho(z,w)$ is minimal. Let $A,B$ be the respective primitive elliptic elements of $\Gamma$ that fix $z,w$ and $\phi, \theta$ be their respective rotation angles. Assume $\theta \leq \phi.$ 

First note that if $\theta = \pi$ then $\phi$ must also equal $\pi.$ But if both $A,B$ have order two, then $\left< A, B \right>$ is elementary (Lemma~\ref{lemMain}~(\ref{lemOrTwo}))

If  $\left< A, B \right>$ is elementary, then by Lemma~\ref{lemMain}~(\ref{lemElem}),  
$$\rho(z,w) \geq \frac{l_{0}}{2},$$ 
where $l_{0}$ is the length of the smallest geodesic of $\Gamma \backslash \hh.$

If $\left< A, B \right>$ is non-elementary, then since $\theta \leq \tfrac{2\pi}{3},$ by Lemma~\ref{lemNonElem} 
$$\sinh(\rho(z,w)) \geq \frac{C}{|\sin(\theta/2)|} \geq C \frac{2}{\sqrt{3}} = 0.152\dots. $$
Thus, in all cases $\rho(z,w) \geq \min\{\tfrac{l_{0}}{2}, \sinh^{-1}(0.152) \}.$
\end{proof}

\begin{rem*}
\emph{A possible application of these ideas lies in the subject of elliptic degeneration of Riemann surfaces. Take a surface with $r$ classes of parabolic cusps. One can find a sequence of compact surfaces, with $r$ classes of elliptic fixed points, that converge to the non-compact surface. In the limit, the elliptic points become the parabolic cusps. The bounds found in this article would imply that the neighborhoods of parabolic cusps can't get close together since none of the elliptic points would have order two. See \cite{Gabin} for more details on elliptic degeneration. } 
\end{rem*}

%%%%%%%%%%%%%%%%%%%%%%
\bibliographystyle{amsalpha} \bibliographystyle{amsalpha}
\bibliography{ellp}

\end{document}